\documentclass[11pt]{amsart}
\usepackage{pdfsync}
\usepackage{enumerate}
\usepackage{graphicx}
\usepackage{latexsym}
\usepackage{amsfonts}
\usepackage[utf8]{inputenc}
\usepackage{amsthm}
\usepackage{amsmath}
\usepackage{amssymb}
\usepackage[
            unicode,
            breaklinks=true,
            colorlinks=true]{hyperref}
\usepackage[active]{srcltx}
\usepackage{amscd}
\usepackage[all]{xy} 
\usepackage{mathrsfs} 
\usepackage{stmaryrd} 
\usepackage{amsopn} 
\usepackage{eepic}
\usepackage{epic}
\usepackage{eucal}
\usepackage{epsfig}
\usepackage{psfrag}

\newtheorem{theorem}{Theorem}[section]

\newtheorem{lemma}[theorem]{Lemma}

\newcommand{\R}{{\mathbb R}}

\newcommand{\Z}{{\mathbb Z}}

\newcommand{\T}{{\mathbb T}}

\newcounter{todo}
\newcommand{\phy}{\varphi}

\newcommand{\DD}{\mathrm{d}}

\renewcommand{\geq}{\geqslant}
\renewcommand{\leq}{\leqslant}

\newcommand{\abs}[1]{\left|#1\right|}
\newcommand{\deriv}[2]{\frac{\partial #1}{\partial #2}}

\newenvironment{remark}{\refstepcounter{theorem}\par\medskip\noindent{\bf
Remark~\thetheorem.}}{\unskip\nobreak\hfill\hbox{ $\oslash$}\par\bigskip}

\newenvironment{question}{\refstepcounter{theorem}\par\medskip\noindent{\bf
Question~\thetheorem}}{\unskip\normalfont \unskip\nobreak\hfill\hbox{\bigskip}}

\newenvironment{definition}{\refstepcounter{theorem}\par\medskip\noindent{\bf
Definition~\thetheorem.}}

\begin{document}

\title{The Hofer question on intermediate symplectic capacities}

\author{\'Alvaro
  Pelayo\,\,\,\,\,\,\,\,\,\,\,\,\,\,\,\,\,\,\,\,\,\,\,\,\,\,San V\~{u}
  Ng\d{o}c}

\maketitle

\begin{abstract}
  Roughly twenty five years ago Hofer asked: \emph{can the cylinder
    ${\rm B}^2(1) \times \mathbb{R}^{2(n-1)}$ be symplectically
    embedded into ${\rm B}^{2(n-1)}(R) \times \mathbb{R}^2$ for some
    $R>0$?} We show that this is the case if $R \geq
  \sqrt{2^{n-1}+2^{n-2}-2}$. We deduce that there are no intermediate
  capacities, between $1$\--capacities, first constructed by Gromov in
  1985, and $n$\--capacities, answering another question of Hofer.  In
  2008, Guth reached the same conclusion under the additional
  hypothesis that the intermediate capacities should satisfy the
  \emph{exhaustion property}.
   \end{abstract}

\section{\textcolor{black}{Introduction}} \label{sec:intro} 
A \emph{symplectic manifold} is a pair $(M,\omega)$ consisting of a 
$2n$\--dimensional ${\rm C}^{\infty}$\--smooth manifold $M$ and a 
\emph{symplectic form} $\omega$, that is, a non\--degenerate closed differential $2$\--form on $M$. For instance, any open subset of $\mathbb{R}^{2n}$ equipped with the
$2$\--form $\omega_0=\sum_{i=1}^{n} {\rm d}x_i \wedge {\rm d}y_i,$
where $(x_1,y_1,\ldots,x_n,y_n)$ denote the coordinates in
$\mathbb{R}^{2n}$, is a symplectic manifold.  If $U$ and $V$ are open
subsets of $\mathbb{R}^{2n}$, a \emph{symplectic embedding} $f \colon U
\to V$ is a smooth embedding such that $f^*\omega_0=\omega_0.$ 
In particular, ${\rm volume}(U) \leq {\rm volume}(V)$.

Let ${\rm B}^{2n}(R)$ denote the open ball of radius $R$ in
$\mathbb{R}^{2n}$, where $R>0$, that is, the set of points
$(x_1,y_1\ldots,x_n,y_n) \in \mathbb{R}^{2n}$ such that $\sum_{i=1}^n
(x_i)^2+(y_i)^2<R^2$.  Gromov's Nonsqueezing\footnote{Many contributions concerning symplectic embeddings followed Gromov's
work, see e.g. Biran \cite{Biran1999, B1, B2}, Ekeland\--Hofer
\cite{EkHo1989}, Floer\--Hofer\--Wyscoki \cite{FlHoWy1994}, Hofer
\cite{Hofer1990}, Lalonde\--Pinsonnault \cite{LP}, McDuff \cite{M2,
  Mc2011}, McDuff\--Polterovich \cite{McPo1994},  
  McDuff\--Schlenk \cite{McSc2012}, and Traynor \cite{Tr1995}.} Theorem \cite{G} states
that there is no symplectic embedding of ${\rm B}^{2n}(1)$ into the
cylinder ${\rm B}^2(R) \times \mathbb{R}^{2(n-1)}$ for $R<1$. \footnote{It is
considered one of the most fundamental results in symplectic topology.
In particular,  it may be used to derive the Eliashberg\--Gromov
Rigidity Theorem   (the theorem says that 
the symplectomorphism group of a manifold is ${\rm C}^0$\--closed
in the diffeomorphism group).}

Coming from the variational theory of Hamiltonian dynamics, Ekeland
and Hofer gave a proof of Gromov's Nonsqueezing Theorem by studying
periodic solutions of Hamiltonian systems.

\subsection{Embeddings} \label{sec:se}

Hofer asked \cite[page~17]{Hofer1990}: is there  $R>0$,
such that the cylinder ${\rm B}^2(1) \times \mathbb{R}^{2(n-1)}$ symplectically embeds into ${\rm B}^{2(n-1)}(R)
\times \mathbb{R}^2$?

\begin{theorem}\label{first} 
  If $n\geq 2$, the cylinder ${\rm B}^2(1) \times \mathbb{R}^{2(n-1)}$
  may be symplectically embedded into the product ${\rm B}^{2(n-1)}(R)
  \times \mathbb{R}^2$ for all $R \geq \sqrt{2^{n-1}+2^{n-2}-2}$.
\end{theorem}

Guth's work
(\cite[Section 2]{Guth2008},  \cite[Section~1]{HiKe2009}) 
  answers the \emph{bounded version} of the question by producing
  symplectic embeddings from ${\rm B}^2(1) \times {\rm B}^{2(n-1)}(S)$, for any $S>0$,
  into ${\rm B}^4(R) \times \R^{2(n-2)}$ for some\footnote{Hind and Kerman afterwards
  showed \cite[Theorems~1.1 and 1.3]{HiKe2009} that such embeddings 
  exist if $R>\sqrt{3}$ and do not exist if
  $R<\sqrt{3}$. The authors settled the case $R=\sqrt{3}$ in \cite{PeVN2013}.} $R>0$.
 The proof of Theorem~\ref{first} builds on  works of
Guth, Hind, Kerman, and Polterovich.

\subsection{Capacities} \label{cap}

Ekeland and Hofer's point of view on Gromov's Nonsqueezing turned out to
be powerful, and allowed them to construct infinitely many new
symplectic invariants, called \emph{symplectic capacities} \cite{EkHo1989, Hofer1990,
  Hofer1990b}. For each integer $1\leq d \leq n$, one can \emph{define} the notion a symplectic
$d$\--capacity (see Section~\ref{sec:capacities}). Whether given $d$,
one can \emph{construct} a symplectic $d$\--capacity is not clear. The
first symplectic $1$\--capacity was constructed by Gromov himself, it
is called the \emph{Gromov radius}:
$$
{\rm c}_{{\rm GR}}(M,\omega):=\sup\{r>0 \,\, |\,\, {\rm exists \,\, symplectic \,\, embedding}\,\, {\rm B}^{2n}(r) \hookrightarrow M\}.
$$
The fact that the Gromov radius is a symplectic $1$\--capacity is
equivalent to Gromov's Nonsqueezing theorem. The
volume induced by the symplectic form provides an example of
symplectic $n$\--capacity. Symplectic $d$\--capacities are called
\emph{intermediate capacities} when $1<d<n$.  In
\cite[page~17]{Hofer1990}, four years after Gromov's work, Hofer
predicts the nonexistence of intermediate capacities.\footnote{Hofer wrote: ``\emph{so far
  no examples are known for intermediate capacities {\rm ($1<d<n$)}.
  It is quite possible that they do not exist."} Hofer continues to
say: \emph{``Some evidence for this possibility is given by the fact
  that there is an enormous amount of flexibility for symplectic
  embeddings $M \hookrightarrow N$ with $\dim M \leq \dim N -2$, see
  Gromov's marvellous book }\cite{Gr1985}".}  We'll prove
the prediction of  Hofer:

\begin{theorem} \label{general0} Let $n\geq 3$.  If $1<d<n$,
  symplectic $d$\--capacities do not exist on any subcategory of the
  category of $2n$\--dimensional symplectic manifolds.
\end{theorem}

The so called $d$\--nontriviality property of $d$\--capacities (see
Section \ref{sec:capacities}, item (3)) cannot be satisfied if $1 < d
< n$ because of Theorem \ref{first}. Hence Theorem~\ref{first} implies
Theorem~\ref{general0}.  Guth  proved  \cite{Guth2008}  
that  intermediate capacities which also satisfy the \emph{exhaustion property}
(the value of the capacity on an open set equals the supremum of the values on
 its compact subsets) do not exist (see Latschev~\cite{La} and Remark~\ref{Guth}).

\begin{remark}
Theorem \ref{general0} implies, in view of Gromov's theorem, that
symplectic $d$\--capacities exist if and only if $d\in \{1,\,n\}$.
Theorem \ref{first} is a
``squeezing statement": an arbitrarily large $\R^N$ may be squeezed
into $\R^2$ provided there is a bounded component whose size can be
increased to make room. This is in agreement with the fact that
$1$\--capacities \emph{exist} due to non\--squeezing,
while $d$\--capacities ($1<d<n$) \emph{do not exist} due to squeezing.
\end{remark}

The literature on the subject is extensive, and we refer to
\cite{CiHoLaSc2007, Gr1985, HoZe1994, Hu2011, Vi1989, Ze2010} and the
references therein.

\section{\textcolor{black}{Symplectic
    capacities}} \label{sec:capacities}

Symplectic capacities were invented in Ekeland and Hofer's influential
paper \cite{EkHo1989,Hofer1990}.  The first capacity, called the
\emph{Gromov radius}, was constructed
by Gromov \cite{G} (its existence follows from the Nonsqueezing Theorem).  
For the basic notions concerning symplectic
capacities we refer to \cite{CiHoLaSc2007}.  We follow the
presentation therein here.  Denote by $\mathcal{E}\ell\ell$ the
category of ellipsoids in $\mathbb{R}^{2n}$ with symplectic embeddings
induced by global symplectomorphisms of $\mathbb{R}^{2n}$ as
morphisms, and by ${\rm Symp}^{2n}$ the category of all symplectic
manifolds of dimension $2n$, with symplectic embeddings as morphisms.
A \emph{symplectic category} is a subcategory $\mathcal{C}$ of ${\rm
  Symp}^{2n}$ such that $(M,\omega) \in \mathcal{C}$ implies that
$(M,\lambda\omega) \in \mathcal{C}$ for all $\lambda>0$.  A
\emph{generalized symplectic capacity} on a symplectic category
$\mathcal{C}$ is a functor $c$ from $\mathcal{C}$ to the
category $([0,\infty],\leq)$ satisfying the following two axioms:
\begin{itemize}
\item[(1)] \emph{Monotonicity}: $c(M,\omega)\leq c(M',\omega')$ if
  there exists a morphism from $(M,\omega)$ to $(M,\omega')$ (this is a reformulation of
  ``functoriality");
\item[(2)] \emph{Conformality}: $c(M,\lambda\omega)=\lambda
  c(M,\omega)$ for all $\lambda>0$.
\end{itemize}
A \emph{symplectic capacity} is a generalized symplectic capacity which, in
addition to (1) and (2), is required to satisfy \emph{nontriviality}:
$$
c({\rm B}^{2n}(1))>0\,\,\,\,\,\, \textup{and}\,\,\,\,\,\,\, c({\rm B}^2(1) \times
\mathbb{R}^{2n-2})<\infty, 
$$
and the \emph{normalization property}
(that is $c({\rm B}^{2n}(1))=1$).  Now let's consider a symplectic
category $\mathcal{C} \subset {\rm Symp}^{2n}$ which contains
$\mathcal{E}\ell\ell$ and let $1 \leq d\leq n$.  A \emph{symplectic
  $d$\--capacity}  on $\mathcal{C}$  is a generalized capacity
satisfying:
\begin{itemize}
\item[(3)] \emph{$d$\--nontriviality}: $c(B^{2n}(1))>0$ and
  \[
  \begin{cases}
    c({\rm B}^{2d}(1) \times \mathbb{R}^{2(n-d)})<\infty\\
    c({\rm B}^{2(d-1)}(1) \times \mathbb{R}^{2(n-d+1)})=\infty\\
  \end{cases}
  \]
\end{itemize}
Symplectic $d$\--capacities are often called \emph{intermediate capacities} if $2
\leq d \leq n-1$. A symplectic $1$\--capacity  is the same as a symplectic capacity.
Intermediate capacities were introduced by Hofer \cite{Hofer1990}
in 1989, but no example has ever been constructed. Hofer
conjectured that it is quite possible that they would not exist. 

\begin{remark} \label{Guth}
The work of Guth \cite[Section 1]{Guth2008}
  implies that  intermediate
  capacities  $c$ which satisfy the \emph{exhaustion property} (Section~\ref{cap}) should not exist.    
  In fact, it is sufficient as Guth indicated that  
    $\lim_{R \to \infty}  c[{\rm B}^{2d}(1) \times {\rm B}^{2(n-d)}(R)] < \infty$ and
   $\lim_{R \to \infty} c[{\rm B}^{2(d-1)}(1) \times {\rm B}^{2(n-d+1)}(R)] =\infty$. The
   proof is analogous to the proof we give of Theorem \ref{general0}.
   \end{remark}

%
%
%
%
%
%
   
   \section{Capacities and embeddings into ${\rm B}^2(R_1) \times {\rm B}^2(R_2)
  \times \mathbb{R}^{2(n-2)}$}

 Let's now consider the following question. As before, let $n \geq 3$.

\begin{question} (Hind and Kerman \cite[Question 3]{HiKe2009}).
  \label{LaGu:queF20} What, if any, are the smallest $0<R_1 \leq R_2$
  such that ${\rm B}^2(1) \times {\rm B}^{2(n-1)}(S)$ may be
  symplectically embedded into ${\rm B}^2(R_1) \times {\rm B}^2(R_2)
  \times \mathbb{R}^{2(n-2)}$?
\end{question}

\vspace{1mm}

Guth's work implies that, for any $R \gneqq \sqrt{2}$ and $S>0$ there
is a symplectic embedding from ${\rm B}^2(1) \times {\rm
  B}^{2(n-1)}(S) $ into ${\rm B}^2(R_1) \times {\rm B}^2(R_2) \times
\mathbb{R}^{2(n-2)}$ (\cite[Theorem 1.6]{HiKe2009}).  Hind and Kerman
proved that for any $0 < R <\sqrt{2}$ there are no symplectic
embeddings of ${\rm B}^2(1) \times {\rm B}^{2(n-1)}(S) $ into ${\rm
  B}^2(R_1) \times {\rm B}^2(R_2) \times \mathbb{R}^{2(n-2)}$ when $S$
is \emph{sufficiently large}.  Their proof is based on a limiting
argument as $\sqrt{2}+\epsilon \to \sqrt{2}$ which may not be directly
applied to the $\sqrt{2}$ case. 

\begin{question}.
  \label{LaGu:queF2} What, if any, are the smallest $0<R_1 \leq R_2$
  such that ${\rm B}^2(1) \times \mathbb{R}^{2(n-1)}$ embeds
  symplectically into ${\rm B}^2(R_1) \times {\rm B}^2(R_2) \times
  \mathbb{R}^{2(n-2)}$?
\end{question}

\vspace{1mm}

The answer to Question \ref{LaGu:queF2} is
given by the following.

\begin{theorem} \label{answers} The product ${\rm B}^2(1) \times
  \mathbb{R}^{2(n-1)}$ embeds symplectically into ${\rm B}^2(R_1)
  \times {\rm B}^2(R_2) \times \mathbb{R}^{2(n-2)}$ with $0<R_1\leq
  R_2$ if and only if $\sqrt{2}\leq R_1$.
\end{theorem}

\subsection*{Idea of proof of Theorem \ref{answers}}

\begin{itemize}
\item[\,] ({\bf Step 1}). We verify that the constructions of
  embeddings which Guth and Hind\--Kerman carried out to answer Question \ref{LaGu:queF20}, and which depends on parameters,
  vary \emph{smoothly} with respect to these parameters.  To do this,
  we follow these authors' constructions with some variations
  checking that at every step there is smooth dependance on the
  parameters involved. This is a priori unclear from the
  constructions, which involve choices of maps, curves, points, etc.
  We overcome this by supplying smooth formulas.  Sometimes we use
  ideas of Polterovich to construct these formulas.
\item[\,] ({\bf Step 2}).  From the smooth family in Step 1, we
  construct a \emph{new} family of smooth symplectic embeddings which
  has, as limit, a symplectic embedding $i$.  We are not claiming that
  $i$ is the ``limit" of the original family (which may not exist).
  The original family is modified according to the upcoming Theorem
  \ref{cor}.
\end{itemize}

\subsection*{Proof of Theorems \ref{first} and \ref{general0}}

\begin{proof}[Proof of Theorem \ref{first}]
  If $1\leq d \leq n-1$, Theorem \ref{answers} may be applied $d-1$
  times to get a symplectic embedding from ${\rm B}^2(1) \times
  \mathbb{R}^{2(n-1)}$ into 
  \[ \underbrace{{\rm B}^2(2^{\frac{1}{2}}) \times {\rm B}^2({2}^1)\ldots \times
  {\rm B}^2(2^{\frac{d-1}2}) \times {\rm B}^2({2}^{\frac{d-1}2})}_{d\,\, \textup{factors}}
  \times \mathbb{R}^{2(n-d)}.
\]
If $d=n-1$, we get a symplectic embedding into $ {\rm
  B}^2(2^{\frac{1}{2}}) \times \ldots \times {\rm
  B}^2(2^{\frac{n-2}2}) \times {\rm B}^2(2^{\frac{n-2}2}) \times
\mathbb{R}^{2} $ which is included in $${\rm
  B}^{2(n-1)}(\sqrt{2+2^2\ldots+2^{n-3}+2^{n-2}+2^{n-2}}) \times
\mathbb{R}^2.$$ The result follows from
$2+2^2\ldots+2^{n-3}+2^{n-2}+2^{n-2}=2^{n-1}+2^{n-2}-2$.
\end{proof}

\begin{proof}[Proof of Theorem \ref{general0}]
Let $1 < d < n$. We have
\[
{\rm B}^{2(d-1)}(1) \times \mathbb{R}^{2(n-d+1)} \subset {\rm B}^{2(d-2)}(1)
\times {\rm B}^{2}(1)\times \R^4 \times \R^{2(n-d-1)}.
\]
By  Theorem \ref{first} ${\rm B}^{2}(1)\times \mathbb{R}^{4}$
embeds into ${\rm B}^{4}(R)\times \mathbb{R}^{2}$ for some
$R>0$. Since ${\rm B}^{2(d-2)}(1)\times {\rm B}^{4}(R)$ is included in
${\rm B}^{2d}(1+R)$ we get a symplectic
embedding ${\rm B}^{2(d-1)}(1) \times \mathbb{R}^{2(n-d+1)} \hookrightarrow
{\rm B}^{2d}(1+R) \times \mathbb{R}^{2(n-d)}$,
which contradicts $d$\--nontriviality.
\end{proof}

\begin{remark}
The radius in Theorem~\ref{general0} is not optimal, as can be seen
for $n=3$ in view of \cite[Theorem~1.2]{PeVN2013}.
\end{remark}

\section{\textcolor{black}{Families with singular limits and
    Hamiltonian dynamics}} \label{sec2}

This section
  has been influenced by many fruitful discussions with Lev Buhovski,
  and we are very grateful to him.

\subsection*{Smooth families}

We start with the following notion of smoothness.

\begin{definition}
  \label{defi:smooth}
  Let $P,M,N$ be smooth manifolds. Let $(B_p)_{p\in P}$ be a family of
  submanifolds of $N$. For each $p\in B_p$, let $\phi_p: B_p \hookrightarrow M$ be
  an embedding.  We say that $(\phi_p)_{p\in P}$ is a \emph{smooth
    family of embeddings} if the following properties hold~:
  \begin{enumerate}
  \item there is a smooth manifold $B$ and a smooth map $g:P\times B\to N$ such
    that $g_p:b\mapsto g(p,b)$ is an immersion and $B_p=g(p,B)$, for
    every $p \in P$;
  \item the map $\Phi:P\times B \to M$ defined by $\Phi(p,b) :=
    \phi_p\circ g(p,b)$ is smooth.
  \end{enumerate}
  In this case we also say that $(\phi_p \colon B_p \hookrightarrow M_p)_{p\in P}$
  is a \emph{smooth family of embeddings} when $M_p$ is a submanifold
  of $M$ containing $\phi_p(B_p)$.  If $M$ and $N$ are symplectic,
  then a \emph{smooth family of symplectic embeddings} is a smooth
  family of embeddings $(\phi_p)_{p\in P}$ such that each
  $\phi_p:B_p\hookrightarrow M$ is symplectic.
\end{definition}

\begin{definition}
  \label{defi:smooth1}
  If in Definition \ref{defi:smooth}, $P$ is a subset of a smooth
  manifold $\tilde{P}$, then we say that the family $(\phi_p)_{p\in
    P}$ is \emph{smooth} if there is an open neighborhood $U$ of $P$
  such that the maps $g:P\times B\to N$ and $\Phi:P\times B \to M$ may
  be smoothly extended to $U \times B$.
\end{definition}

\subsection*{Limits of smooth families}

We present a construction to remove a singular limit of a smooth
family, see Figure \ref{fig:Wt} for an illustration of the theorem. A  related statement is \cite[Corollary~1.2]{Mc1991}.

\begin{theorem} \label{cor} Let $N$ be a symplectic manifold, and let
  $W_t\subset N$, $t \in (0,\,a)$, be a family of simply connected open subsets with
  $\overline{W_s} \subset W_t$, for $s,t \in (0,\,a)$ and $t<s$.  Let
  $$W_0:=\bigcup_{t \in (0,\,a)} W_t.$$ Let $$(\phi_t \colon W_t \hookrightarrow
  M)_{t \in (0,\,a)}$$ be a smooth family of symplectic embeddings such
  that for any $t,s>0$, the set $\bigcup_{v\in [t,s]}\phi_v(W_v)$ is
  relatively compact in $M$.  Then there is a symplectic embedding
  $W_0 \hookrightarrow M$.
\end{theorem}

\begin{figure}[h]
  \centering \centering
  \includegraphics[width=1\textwidth]{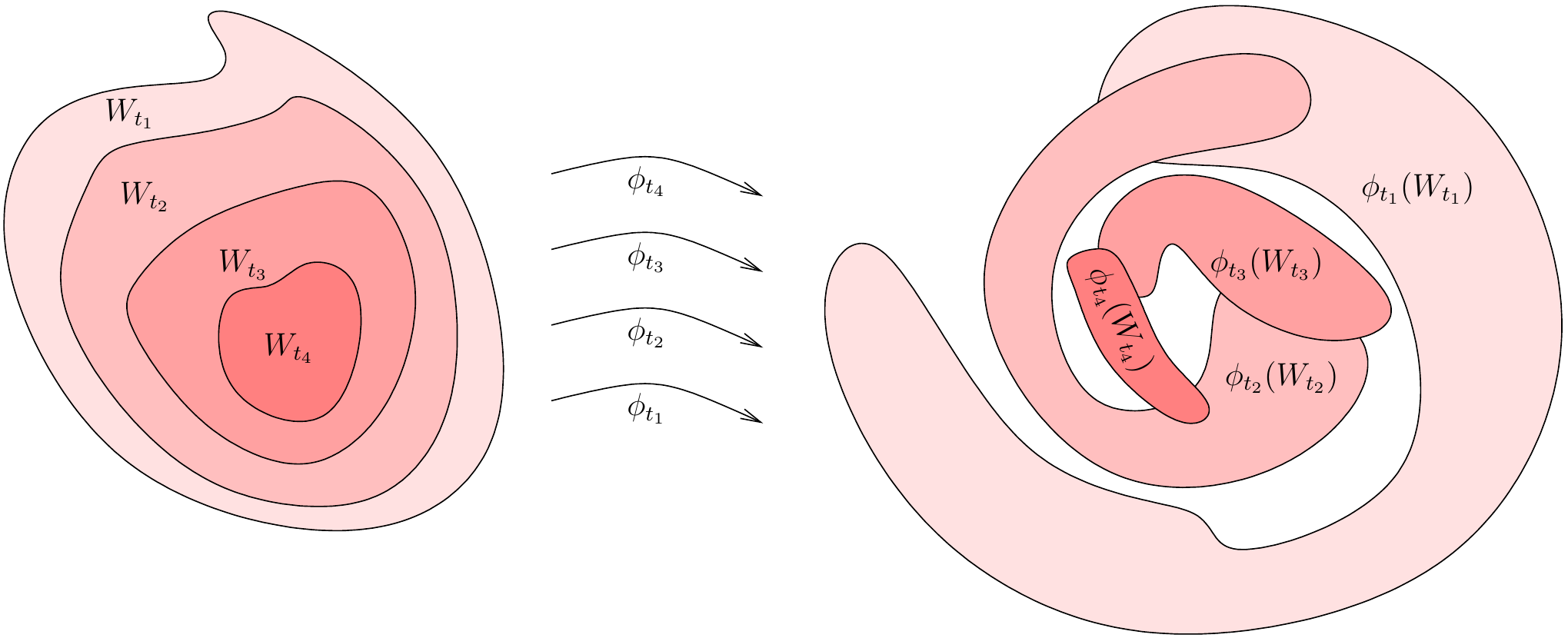}
  \caption{The figure illustrates the hypothesis of
    Theorem~\ref{cor}. Here $t_1<t_2<t_3<t_4$. The theorem does
    \emph{not} say that the family of embeddings $(\phi_t \colon W_t
    \hookrightarrow M)_{t \in (0,\,a)}$ has a limit as $t\to
    0$. Actually, the images $\phi_t(W_t)$ may overlap in complicated
    ways, be disjoint etc. An embedding on the union of all the $W_t$
    can be constructed using Hamiltonian flows to modify the family,
    see the proof of the result.}
  \label{fig:Wt}
\end{figure}

\begin{remark} \label{newest} If $N=\mathbb{R}^{2n}$ and if there is a
  continuous function $ r \colon (0,a) \subset \mathbb{R} \to
  (0,\infty)$, $v \mapsto r(v)$, such that $\phi_v(W_v)$ is contained
  in ${\rm B}^{2d}(r(v))$ for every $v \in (0,a)$, then the hypothesis
  in Theorem \ref{cor} that for any fixed $t,s>0$, the set
  $\bigcup_{v\in [t,s]}\phi_v(W_v)$ is relatively compact in $M$, is
  automatically satisfied. 
  Indeed, we have
  that $ \overline{ \bigcup_{v\in [t,s]}\phi_v(W_v)} \subset
  \overline{{\rm B}^{2d}(\max_{t \in [t,s]}r(v))}.  $
\end{remark}

\subsection*{Key lemmas}

We use two lemmas in order to prove Theorem \ref{cor}.

\begin{lemma} \label{Lemma1} Let $W_t\subset N$, $t \in (0,a)$, be a family of simply
  connected open subsets of a symplectic manifold $N$.  Let $(\phi_t
  \colon W_t \hookrightarrow M)_{t \in (0,a)}$ be a smooth family of symplectic embeddings
  such that:
  \begin{itemize}
  \item[{\rm (i)}] $\overline{W_s}\subset W_t$ when $t < s$;
  \item[{\rm (ii)}] for any $t,s>0$, the set $\bigcup_{v\in
      [t,s]}\phi_v(W_v)$ is relatively compact in $M$.
  \end{itemize}
  Then for any $t<t'<s$ there exists a smooth time-dependent
  Hamiltonian $G_v:M\to\R$, $v\in[t,t']$, whose Hamiltonian flow
  $\psi_v$ starting at $v=t$ is defined for all $v\in[t,t']$ and
  satisfies $ \psi_{t'}\circ \phi_t |_{W_s} = \phi_{t'} |_{W_s}.  $
\end{lemma}
\begin{proof}
  We divide the proof into three steps.
  \\
  \\
  \emph{Step 1}.  Let $s''\in (t',s)$. Let $x\in W_{s''}$. By
  \textup{(i)} we have that for any $v\in [t,t']$, $x\in W_v$.
  Therefore one can take the following derivative~:
  \[
  X_v(\phi_v(x)) := \deriv{\phi_v(x)}{v},
  \]
  which defines a vector field on $\phi_v(W_{s''})$. Because all
  $\phi_v$'s are symplectic, the time-dependent vector field $X_v$ is
  symplectic. Hence the pull-back $(\phi_v)^*X_v$ is symplectic. Since
  $W_{s''}$ is simply connected, $(\phi_v)^*X_v|_{W_{s''}}$ is
  Hamiltonian~: there exists a smooth function $(x,v)\mapsto
  \tilde{H}_v(x)$ on $W_{s''}\times [t,t']$ such that
  $$
  \iota_{(\phi_v)^*{X_v}}\omega = - \DD \tilde{H}_v.
  $$
  We let $$H_v(y):=\tilde{H}_v(\phi_v^{-1}(y)),$$ which is a Hamiltonian
  function defined on $\phi_v(W_{s''})$ for the vector field $X_v$.
  This concludes Step 1.
  \\
  \\
  \emph{Step 2}. We'll construct a smooth family $(\tau_v:M\to \R)_{v\in [t,t']}$, with~: \begin{align}
    \label{equ:cutoff1}
    \tau_v |_{\phi_v(W_s)} \equiv 1;\\
    \label{equ:cutoff2}
    \tau_v |_{M\setminus \phi_v(W_v)} \equiv 0.
  \end{align}
  In order to do this, fix $s'\in (s'',s)$, and let
  $\chi\in\textup{C}^\infty(M)$ be equal to $1$ on $W_s$ and to $0$ on
  $M\setminus W_{s'}$. We simply define
  \[
  \tau_v(y) :=
  \begin{cases}
    \chi\circ\phi_v^{-1}(y) & \text{ if } y\in\phi_v(W_v)\\
    0 & \text{ otherwise.}
  \end{cases}
  \]
  The map $\tau_v$, for $v\in[t,t']$, satisfies~\eqref{equ:cutoff1}
  and~\eqref{equ:cutoff2}. It remains to see that $ (v,y)\mapsto
  \tau_v(y) $ is smooth. First, let $(v_0,y_0)\in [t,t']\times M$ be
  such that $y_0\in \phi_{v_0}(W_{v_0})$. Using the continuity of the
  family $(\phi_v)$ and the fact that $\phi_{v_0}(W_{v_0})$ is open in
  $M$, we see that $y\in \phi_{v}(W_{v})$ for $(v,y)$ in a small open
  neighborhood of $(v_0,y_0)$. Hence, in this neighborhood,
$$(v,y)\mapsto \tau_v(y)=\chi\circ\phi^{-1}(v)$$ is smooth. Second,
suppose that $y_0\not\in \phi_{v_0}(W_{v_0})$.  Therefore $y_0\not\in
\phi_{v_0}(\overline{W_{s'}})$, and the latter being closed, there is
a small neighborhood of $(v_0,y_0)$ in which all $(v,y)$ satisfy
$y\not\in \phi_{v}(\overline{W_{s'}})$. Hence
$y\not\in\phi_v(W_{s'})$, and both cases in the definition of $\tau_v$
lead to $\tau_v(x)=0$, which proves the smoothness.
\\
\\
\emph{Step 3}. We may now define a smooth time dependent Hamiltonian $G
\colon [t,\,t'] \times M \to \mathbb{R}$ by $$G_v(y)=H_v(y) \tau_v(y)$$ 
for $y \in \phi_v(W_v)$ and $ G_v(y)=0 $ for $y \in M \setminus
\phi_v(W_v)$. Of course, $G_v=H_v$ on $\phi_v(W_s)$.

Let $Y_v$ be the Hamiltonian vector field associated to $G_v$ and let
$\psi(v,y)$ be the flow of $Y_v$ starting from time $t$~:
\[
\begin{cases}
  \displaystyle \deriv{\psi(v,y)}{v} = Y_v(\psi(v,y)) \\
  \psi(t,y)=y.
\end{cases}
\]
The vector field $Y_v$ vanishes outside of the fixed set $
\bigcup_{v\in [t,t']}\phi_v(W_v), $ which is relatively compact in $M$
for any fixed $t,t'>0$ by assumption (ii).  This implies that the flow
$\psi(v,y)$ can be integrated up to time $t'$.

Let
  $$\phy(v,x):=\psi(v,\phi_t(x)).$$ Then $\phy$ satisfies the Cauchy
  problem on $W_s$~:
  \[
  \begin{cases}
    \displaystyle \deriv{\phy(v,x)}{v} = Y_v(\phy(v,x)) = X_v(\phy(v,x))\\
    \phy(t,x) = \phi_t(x).
  \end{cases}
  \]
  Therefore, for all $x\in W_s$, $\phy(v,x)=\phi_v(x).$ In particular,
  when $v=t'$, we get, with $\Psi(y):=\psi(t',y)$,
  \[
  \Psi \circ \phi_t |_{W_s} = \phi_{t'} |_{W_s}
  \]
  This concludes the proof.
\end{proof}

\begin{lemma} \label{Lemma2} Let $V_1 \subset V_2 \subset \ldots
  \subset V_n \subset \ldots$ be a sequence of simply connected open
  subsets of a symplectic manifold $N$. Let $i_n \colon V_n \hookrightarrow M$, $n \in \mathbb{N}^*$, be
  a sequence of symplectic embeddings into another symplectic manifold $M$ 
  such that for any $n\geq 2$
  there exists a symplectomorphism $\psi_n \colon M \to M$ satisfying
  $$
  \psi_n \circ i_{n+1}|_{V_{n-1}} =i_n|_{V_{n-1}}. $$ Denote
  $$V:=\bigcup_{n=1}^{\infty}V_n.$$ Then there exists a symplectic
  embedding $j \colon V \hookrightarrow M$.
\end{lemma}
\begin{proof}
  Define $j \colon V \to M$ by
$$j(x):=\psi_2 \circ \psi_3 \circ \ldots \circ \psi_{n-1} \circ i_n(x)$$ for
$x \in V_{n-1} \subset V$ for $n>2$. This definition is independent of
the choice of $n>2$ for which $x \in V_{n-1}$. Then $j$ is a local
symplectomorphism which is injective (any two points $x,y$ are
contained in a common $V_{n-1}$); thus it is a symplectic embedding.
\end{proof}

\subsection*{Proof of Theorem \ref{cor}}
Consider the sequence of domains $V_n:=W_{1/n}.$ For each $n\geq 3$,
consider the family
\[
\left\{W_t \,\, | \,\,
  t\in{\textstyle\left(\frac{1}{n+2},\frac{1}{n-2}\right)}\right\},
\]
and the values $s=\frac{1}{n-1}, \,\,\,\,\, t=\frac{1}{n+1}\,\,\,\,\,
\textup{and}\,\,\,\,\, t'=\frac{1}{n}.$ Then Lemma~\ref{Lemma1} gives
us a symplectomorphism $\psi_n:M\to M$ such that
\[
\psi_n \circ i_{n+1}|_{V_{n-1}} =i_n|_{V_{n-1}},
\]
which is the assumption of Lemma~\ref{Lemma2}. Since $\bigcup_{n\geq
  3} V_n = \bigcup_{t \in (0,\,a)} W_t,$ we get Theorem \ref{cor}.

\section{\textcolor{black}{Guth's Lemma for families}} \label{sec2a}

The following statement is a smooth family version (see Definition~\ref{defi:smooth}) 
of the Main Lemma in Guth \cite[Section 2]{Guth2008}.  As before, $n\geq 3$.

\begin{lemma} \label{pp:10} 
 Let
$\Sigma$ be the symplectic torus
$\mathbb{T}^2=\mathbb{R}^2/\mathbb{Z}^2$ of area $1$ minus the
``origin'' (\emph{i.e.} minus the lattice $\Z^2$, $\Sigma=(\R^2 \setminus \Z^2)/\Z^2$). There is a smooth
  family $(i_{R})_{R>1/3}$ of symplectic embeddings $i_{R} \colon {\rm
    B}^{2(n-1)}(R) \hookrightarrow \Sigma \times {\rm
    B}^{2(n-2)}(10R^2).$
\end{lemma}
In order to verify Lemma \ref{pp:10} we need to explain why the
construction in \cite{Guth2008} depends \emph{smoothly} on the
parameter $R\in (1/3,\infty)$.  Checking this amounts to checking that
the ``choices" therein made depend smoothly on $R$ and
$\epsilon$. Guth's Lemma is valid for $R=1/3$, however we shall see
that the family is not smooth at this value (there is a square root
singularity).
\begin{proof}
  We may restrict to $n=3$ with a smaller constant~: $ \textup{B}^4(R)
  \hookrightarrow \Sigma \times \textup{B}^2(\sqrt{72}R^2).  $ Indeed,
  on the left hand-side we use the natural embedding
  $\textup{B}^{2(n-1)}(R)\subset B^{2(n-3)}(R)\times \textup{B}^4(R)$,
  and on the right hand-side we use the natural embedding
  $\textup{B}^{2(n-3)}(R)\times \textup{B}^2(\sqrt{72}R^2)\subset
  \textup{B}^{2(n-1)}(\sqrt{72R^4 + R^2})$ and notice that
  $\sqrt{72R^4+R^2}\leq 10R^2$, in view of $R>1/3$. To check
  smoothness with respect to $R$ we need to write some
  explicit formulas for maps and domains which were not explicitly
  written in Guth's paper. Then the smoothness with respect to $R$ 
  as in Definition \ref{defi:smooth} becomes equivalent to
  the smoothness of the formulas. In terms of the notation in
  Definition~\ref{defi:smooth}, we let $B=\textup{B}^{4}(1)\subset
  N=\R^4$, $P=(1/3,\infty)$, the map $g$ is just a scaling~:
  $g(R,b)=R \cdot b$, $R\in P$, $b\in B$, and $M=\Sigma\times \R^2$.  Guth's
  proof has two steps. The first one, due to Polterovich, is to
  construct a linear symplectic embedding of $\textup{B}^4(R)$ into
  $\T^2 \times \textup{B}^2(\sqrt{72}R^2)$, when $R\geq 1/3$. The
  second step is to modify this embedding by a nonlinear
  symplectomorphism in order to avoid a point in $\T^2$. Both steps
  depend on the radius $R$, therefore we have to check the smooth
  dependence.
  \\
  \\
  \emph{Step 1}.  We want a plane $V_R \subset \mathbb{R}^4$, depending
  smoothly on $R$, such that
  \begin{equation}
    \int_{{\rm B}^4(R) \cap V_R} \omega=\frac{\pi}{9}.
    \label{equ:VR}
  \end{equation}
  It turns out that one can give an easy formula for this plane. For
  $t>0$, let $ W_t:={\rm span}\{ (1,0,0,0),\,e_t \} $ with
  $e_t:=(0,t,1-t,0)$. Let $\varphi_t \colon \mathbb{R}^2 \to W_t$ be
  the linear parameterization given by
  $\varphi(u,v):=(u,tv,(1-t)v,0).$ We have that $\varphi_t^*\omega=t
  {\rm d}v \wedge {\rm d}u$.  On the one hand,
  \begin{eqnarray} {\rm B}^4(R) \cap W_t&=&\Big\{ (x_1,y_1,x_2,y_2)
    \in V_t\,\,|\,\,
    (x_1)^2+(y_1)^2+(x_2)^2+(y_2)^2\leq R^2  \Big\} \nonumber \\
    &=& \varphi_t \left(\Big\{(u,v)\in\R^2\, | \,
      \frac{u^2}{R^2}+\frac{v^2}{\frac{R^2}{2t^2-2t+1}} \leq
      1\Big\}\right), \nonumber
  \end{eqnarray}
  which is the image of an ellipse of area $\pi a b$, where $a=R$,
  $b=\frac{R}{\sqrt{2t^2-2t+1}}$. Therefore
$$
\int_{\varphi_t^{-1}(\textup{B}^4(R) \cap W_t)} t {\rm d}v \wedge {\rm
  d}u=\frac{t\pi R^2}{\sqrt{2t^2-2t+1}}.
$$
If $R>1/3$, the equation $\frac{tR^2}{\sqrt{2t^2-2t+1}}=\frac{1}{9}$
has two solutions, and one of them is a smooth positive function
$(1/3,\infty)\ni t\mapsto R(t)$.  Thus, we may let $V_t:=W_{R(t)},$
and we satisfy~\eqref{equ:VR}.

Now, let $(f_{1,R},f_{2,R})$ be an orthonormal basis of $V_R$, and
$(f_{1,R}^{\perp},f_{2,R}^{\perp})$ be a symplectic basis of the
symplectic orthogonal complement $V_R^\perp$. These basis may be
chosen to depend smoothly on $R$. Then the linear map
$$
(L_R)^{-1} \colon (x_1,y_1,x_2,y_2) \mapsto \lambda x_1f_{1,R}+
y_1f_{2,R}+x_2f_{1,R}^{\perp}+y_2f_{2,R}^{\perp}
$$
is symplectic when $\omega(\lambda f_{1,R}, f_{2,R})=1$,
i.e. $\lambda=(\omega(f_{1,R},f_{2,R}))^{-1}.$ Thus $(L_R)_{R>1/9}$ is
a smooth family of linear symplectomorphisms that map planes parallel
to $V_R$ to planes parallel to the $(x_1,y_1)$\--plane, and maps disks
parallel to $V_R$ to disks. Let $P$ be an affine plane parallel to the
$(x_1,y_1)$\--plane and let $\tilde{P}_R=L_R^{-1}(P)$.  Then ${\rm
  B}^4(R) \cap \tilde{P}_R$ is a ball of radius $\leq R$ parallel to
${\rm B}^4(R) \cap V_R$. Therefore, since $\omega$ is invariant by
translation, we have that $ \int_{{\rm B}^4(R) \cap \tilde{P}_R}
\omega \leq \int_{{\rm B}^4(R) \cap V_R} \omega=\frac{\pi}{9}.  $
(note that ${\rm B}^4(R) \cap \tilde{P}_R$ can be translated to be a
subset of ${\rm B}^4(R) \cap V_R$.) We know that $L_R({\rm B}^4(R)
\cap \tilde{P}_R)$ must be a Euclidean disk $\textup{B}^2(r)$ in the 
$(x_1,y_1)$\--plane. Since $L_R$ is symplectic, we have that
$$\int_{L_R({\rm B}^4(R)) \cap P} \omega =
\int_{\underbrace{\scriptstyle L_R({\rm B}^4(R) \cap
    \tilde{P}_R)}_{{\rm B}^2(r)}} \omega = \int_{{\rm B}^4(R) \cap
  \tilde{P}_R} \omega \leq \frac{\pi}{9},$$ and therefore $\int_{{\rm
    B}^2(r)} {\rm d}x_1 \, {\rm d}y_1\leq \frac{\pi}{9},$ and hence $r
\leq 1/3$.  $L_R(\textup{B}^4(R))$ is an ellipsoid (by this we mean
the open set bounded by the ellipsoid) in $\R^4$ whose half\--axes are
smooth, positive functions of $R$. Hence its projection onto the
$(x_2,y_2)$-plane is an ellipse with the same properties. Therefore,
there exists a smooth positive function $\mu(R)$ such that the
projection onto the $(x_2,y_2)$-plane of $\ell_R\circ
L_R(\textup{B}^4(R))$ is a disk, where $\ell_R$ is the
symplectomorphism $$(x_1,y_1,x_2,y_2)\mapsto
(x_1,y_1,\mu(R)x_2,(\mu(R))^{-1}y_2).$$ In the sequel, we assume that
$L_R$ is this new symplectomorphism $\ell_R\circ L_R$.

The rest of the proof of Step 1 does not involve any construction
depending on $R$. We repeat the argument here for the sake of
completeness. Let $Q:\R^2\times \R^2\to \T^2 \times \R^2$ be the map
defined as the quotient map $\R^2\to\T^2$ on the first factor,
and the identity on the second one (see Figure~\ref{fig:ellipsoid}).
\begin{figure}[h]
  \centering \centering
  \includegraphics[width=0.8\textwidth]{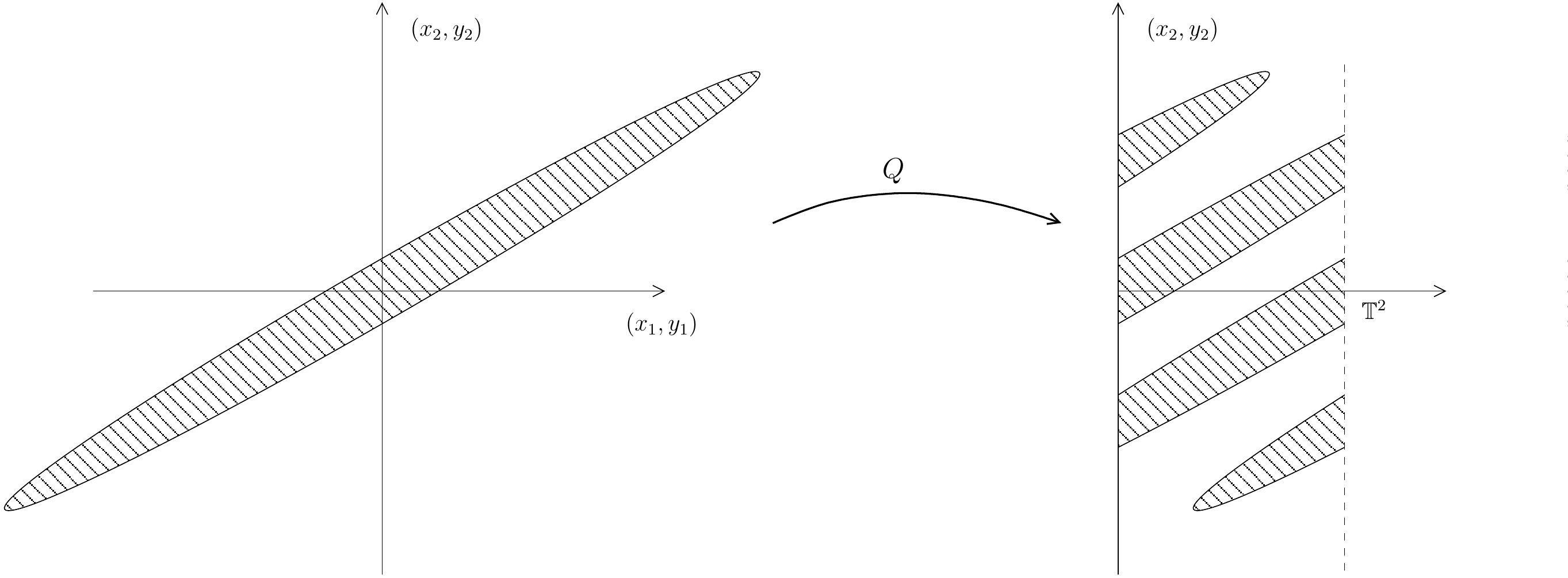}
  \caption{Vertical slices of the ellipsoid.}
  \label{fig:ellipsoid}
\end{figure}
$Q$ restricted to the ellipsoid $L_R(\textup{B}^4(R))$ is injective~:
indeed the ``vertical'' coordinates $(x_2,y_2)$ are preserved, and the
intersection of $L_R(\textup{B}^4(R))$ with a horizontal plane is a
disk of radius $\leq \frac{1}{3}<\frac{1}{2}$. Hence
$Q|_{L_R(\textup{B}^4(R))}$ is an embedding.  The desired embedding is
$Q \circ L_R$, which depends smoothly on $R$.  Let $\pi_2:\R^2\times
\R^2 \to \R^2$ be the projection onto the second factor.  It remains
to estimate the size of $\pi_2(L_R(\textup{B}^4(R)))$, which we know
is an open disk. Let $S$ be the radius of this disk, and let
$p=(x_2,y_2)$ be a point in the concentric disk of radius $S/2$.
Because of \eqref{equ:VR}, the preimage $\pi_2^{-1}(0,0)$ is a disk of
radius $1/3$. Since the ellipsoid is convex, the preimage
$\pi_2^{-1}(p)$ (which is a disk) must have a radius at least $1/6$.
We can now get a lower bound for the volume $v$ of
$L_R(\textup{B}^4(R))$ by integrating over the subset which projects
onto $x_2^2+y_2^2\leq S^2/4$~: $v\geq \pi\frac{S^2}{4}\frac{\pi}{36}$.
Since $L_R$ is symplectic and hence volume preserving, $v$ is also the
volume of $\textup{B}^4(R)$~: $v=\frac{1}{2}\pi^2R^4$. Hence $S\leq
\sqrt{72}R^2$.
\\
\\
\emph{Step 2}. We want now to modify $L_R$ by a nonlinear
symplectomorphism $\tilde{\Psi}_R$, such that the image
$\tilde{\Psi}_R\circ L_R(\textup{B}^4(R))$ avoids the integer lattice
$\Z^2\times\R^2$. Then the required embedding will simply be
$Q\circ\tilde{\Psi}_R\circ L_R$.

We are not going to repeat Guth's argument, but simply to point out
the smooth dependence on $R$.  Let $\pi_1:\R^2\times \R^2 \to \R^2$ be
the projection onto the first factor.  Let $\rho(R)\geq 1$ be a smooth
function such that the ellipse $\pi_1(L_R(\textup{B}^4(R)))$ is
contained in the disk of radius $\rho(R)$. For instance one can take
$\rho(R)$ to be $1$ plus the sum of the two half\--axes of the ellipse.
Then $\tilde{\Psi}_R=\Psi_R\otimes \textup{Id}_{\R}$, where $\Psi_R$
is a symplectomorphism of $\R^2$, obtained by lifting a diffeomorphism
$\Phi_R$ of the $x_1$ variable. We define $\Phi_R(x_1) =
x_1+f_R(x_1),$ where $f_R \colon \mathbb{R} \to \mathbb{R}$ is a smooth function
that satisfies the following properties~:
\begin{enumerate}
\item $f'_R(x_1)\geq - 0.1$;
\item $f_R$ is periodic of period $1$;
\item $f_R(k)=0$, $\forall k\in\Z$;
\item $f'_R(k)=100\rho(R)$, $\forall k\in\Z$;
\item $\abs{f_R}\leq 10^{-4}$;
\item The map $(1/3,\infty)\times\R \ni (R,x_1)\mapsto f_R(x_1)$ is
  smooth.
\end{enumerate}
A function satisfying these requirements is depicted in
Figure~\ref{fig:function_f}.
\end{proof}
\begin{figure}[h]
  \centering
  \includegraphics[width=0.8\textwidth]{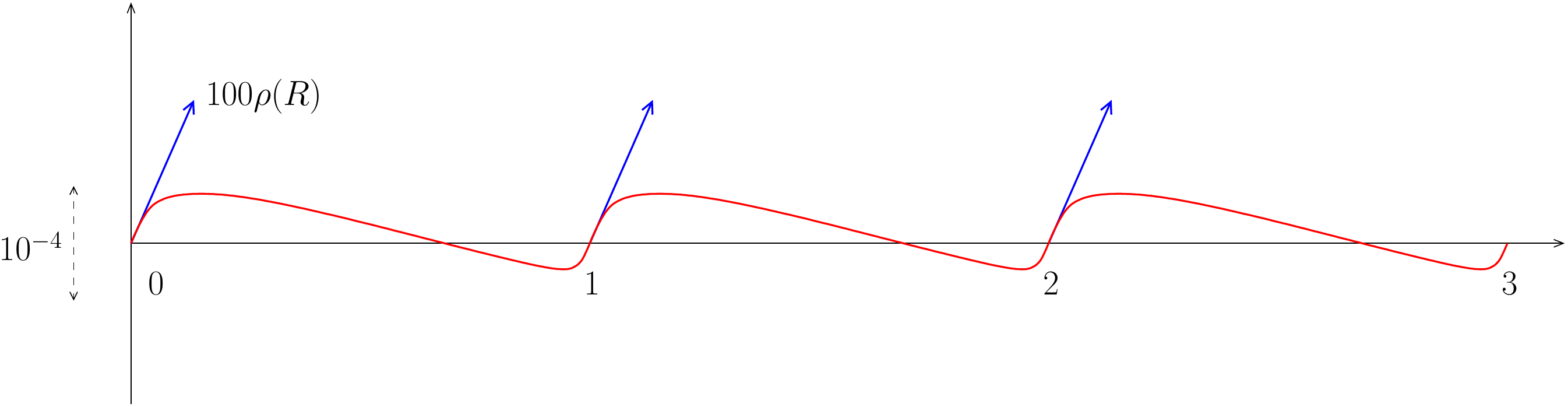}
  \caption{function $f_R$.}
  \label{fig:function_f}
\end{figure}

\section{\textcolor{black}{Embeddings into ${\rm B}^2(R) \times {\rm
      B}^2(R) \times \mathbb{R}^{2(n-2)}$}} \label{sec4}

We start with a particular case 
of the classical non\--compact Moser theorem (see also \cite[Theorem~B.1, Appendix]{Schlenk2005}):

\begin{lemma}[Greene and Shiohama,  Theorem~1 in \cite{GrSh1979}] \label{ncmoser0} 
If $\Sigma$ is a connected
  oriented $2$\--manifold and if $\omega$ and $\tau$ are area forms on
  $\Sigma$ which give the same finite area, then there is a
  symplectomorphism $\varphi \colon (\Sigma,\omega) \to (\Sigma,\tau)$.
\end{lemma}

The Greene\--Shiohama result remains valid when varying with respect to smooth parameters.

\begin{lemma} \label{ncmoser} Let $I \subset \mathbb{R}$ be an
  interval. Let $\{M_{\delta}\}_{\delta \in I}$ and
  $\{N_{\delta}\}_{\delta \in I}$ be smooth families of connected
  $2$\--manifolds such that on each $M_{\delta}, N_{\delta}$ there are
  area forms $\omega_{\delta},\tau_{\delta}$, respectively, giving the same finite
  area for each $\delta \in I$.  Then there is a smooth family of
  symplectomorphisms $(\varphi_{\delta} \colon M_{\delta} \to
  N_{\delta})_{\delta \in I}$.
\end{lemma}

\begin{figure}[h]
  \centering
  \includegraphics[width=0.7\textwidth]{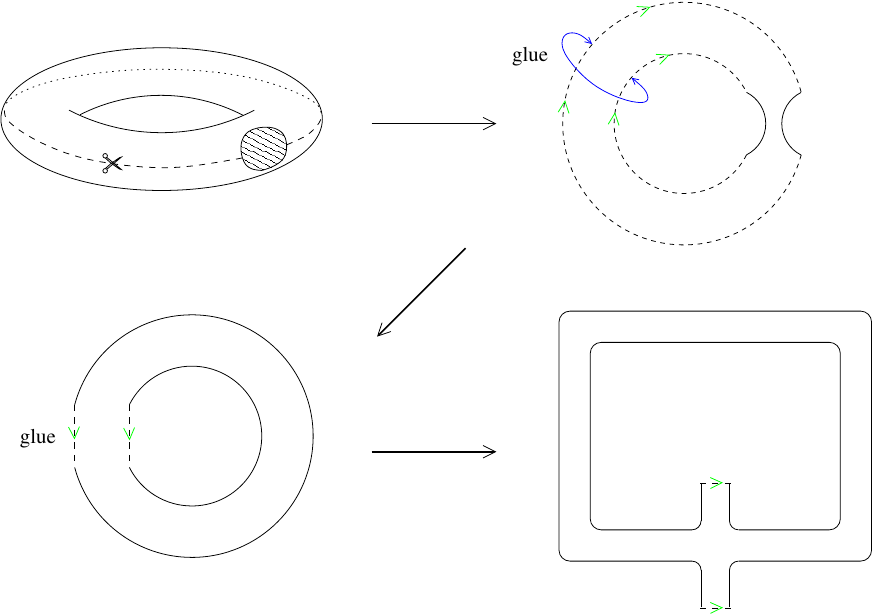}
  \caption{Steps to arrive at Figure \ref{fig:fancyimmersion}.}
  \label{fig:steps}
\end{figure}
The following is a smooth family version of the main statement proven
by Hind and Kerman in \cite[Section~4.1]{HiKe2009}.  It concerns ball
embeddings constructed using Hamiltonian flows.  As before, let
$\Sigma$ be the symplectic torus
$\mathbb{T}^2=\mathbb{R}^2/\mathbb{Z}^2$ of area $1$ minus the
``origin'' (\emph{i.e.} minus the lattice $\Z^2$). 

\begin{theorem} \label{pp:11} 
For any $\epsilon>0$,
we let
$\Sigma(\epsilon):=(\R^2\setminus\sqrt\epsilon\Z^2)/\sqrt\epsilon\Z^2$
be the scaling of $\Sigma$ with symplectic area $\epsilon$.  
  There exist constants $\epsilon_0>0$, $c>0$,
  and a smooth family $({I}_{\epsilon})_{\epsilon \in (0,\epsilon_0]}$
  of symplectic embeddings
 $
  {I}_{\epsilon} \colon \Sigma(\epsilon) \times {\rm B}^2(1)
  \hookrightarrow {\rm B}^2(\sqrt 2 + c\epsilon) \times {\rm
    B}^2(\sqrt 2).
  $
\end{theorem}

\begin{proof}
 We have organized the proof in several steps. 
As in the proof of Lemma \ref{pp:10}, smoothness is the sense of 
Definition  \ref{defi:smooth}.
  \\
  \\
  \emph{Step 1} (\emph{Definition of immersion $i_{\epsilon}$}). For
  sufficiently small fixed $\epsilon>0$ we may define a smooth
  immersion
  \begin{eqnarray} \label{theone} i_{\epsilon} \colon
    \Sigma(\tilde{\epsilon}) \hookrightarrow \mathbb{R}^2,
  \end{eqnarray}
  where
  \begin{eqnarray} \label{useful_formula}
    \tilde{\epsilon}:=100\epsilon
  \end{eqnarray} by Figure \ref{fig:fancyimmersion}, with $a=\epsilon^2$.  In particular,
  the double points of the immersion are concentrated in the small
  region $[-a,a]\times[-\epsilon/2,\epsilon/2]$.
  The topological steps to transform the punctured torus
  $\Sigma(\tilde\epsilon)$ into such a domain are depicted in
  Figure \ref{fig:steps}.
  \\
  \\
  \emph{Step 2} (\emph{Modifying $i_{\epsilon}$ to make it symplectic}).
  By Moser's argument applied to
  $(\Sigma(\tilde\epsilon),\omega_{\Sigma(\tilde\epsilon)})$ and
  $(\Sigma(\tilde\epsilon),i_\epsilon^*\omega_0)$
  (Lemma~\ref{ncmoser}),  where $\omega_0$ is the standard symplectic
  form on $\mathbb{R}^2$,  the immersion 
  (\ref{theone}) may be modified so as to obtain a  symplectic immersion.
  For this to hold, we need that 
  \begin{eqnarray} \label{moserequality}
    \int_{\Sigma(\tilde{\epsilon})} i_{\epsilon}^*\omega_0={\rm
      area\,\, of}\,\, \Sigma(\tilde{\epsilon}).
  \end{eqnarray}
 
   \begin{figure}[h]
    \centering
    \includegraphics[width=0.85\textwidth]{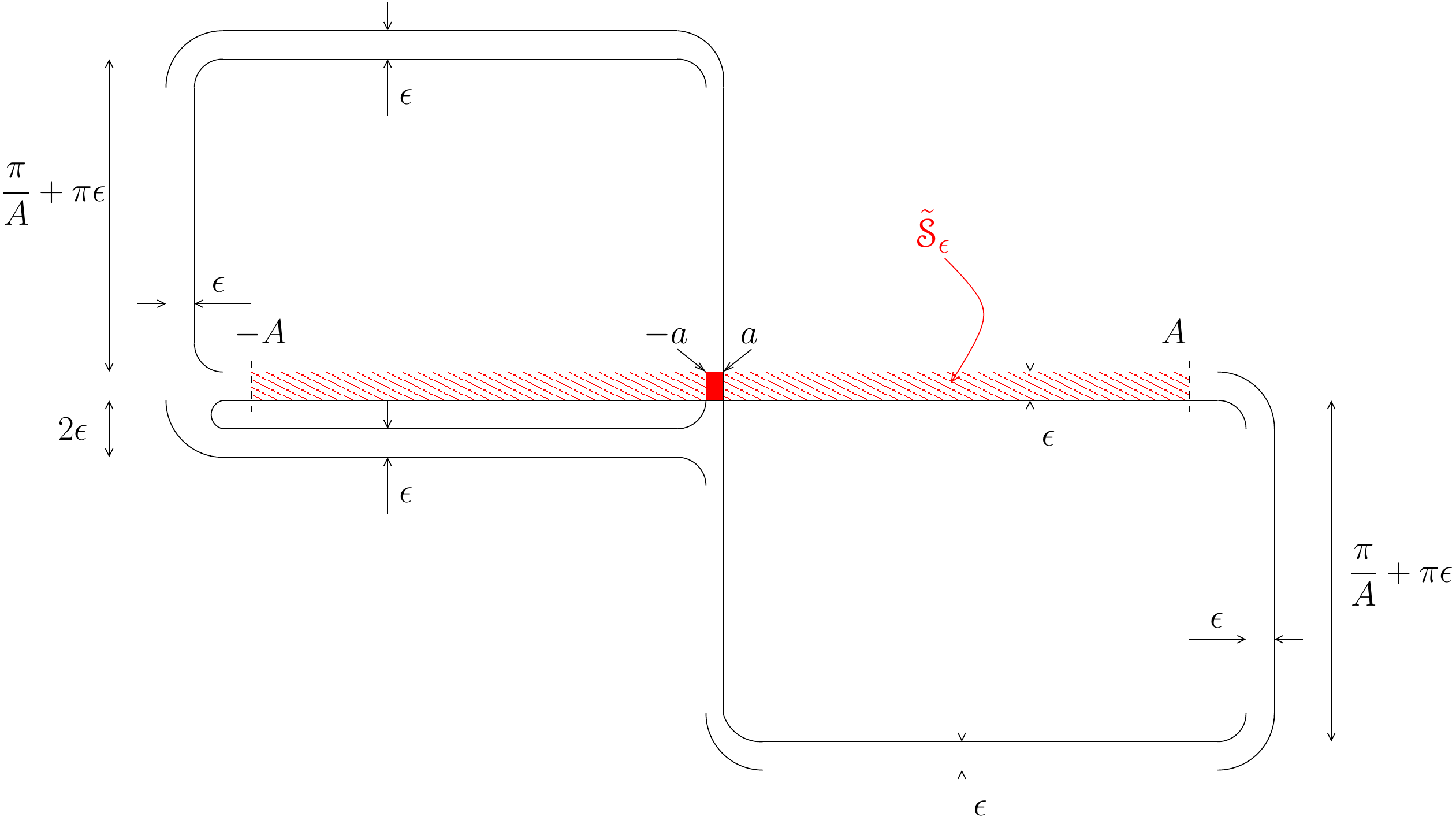}
    \caption{The immersion $i_{\epsilon} \colon
      \Sigma(\tilde{\epsilon}) \hookrightarrow \mathbb{R}^2$.}
    \label{fig:fancyimmersion}
  \end{figure}
  The right hand side of (\ref{moserequality}) is equal to
  $\tilde{\epsilon}=100\epsilon$ by definition of
  $\Sigma(\tilde{\epsilon})$.  Let's compute the left hand side of
  (\ref{moserequality}).  {From} Figure~\ref{fig:fancyimmersion}, it
  is equal, for sufficiently small $\epsilon>0$, to the sum of the
  areas of five horizontal rectangles of area $A\epsilon$, and four
  vertical rectangles of area $\frac{\pi\epsilon}{A-a}$, plus several
  corner squares whose area is a smooth function of $\epsilon$ of
  order $\epsilon^2$. Since $a=\epsilon^2$ we obtain that, for
  sufficiently small $\epsilon>0$~:
  \begin{eqnarray} \label{stf} \int_{\Sigma(\tilde{\epsilon})}
    i_{\epsilon}^*\omega_0= 5A \epsilon + \frac{4\pi
      \epsilon}{A}+\mathcal{O}(\epsilon^2).
  \end{eqnarray}
  Let $p(A)= \int_{\Sigma(\tilde{\epsilon})} i_{\epsilon}^*\omega_0$,
  so that (\ref{moserequality}) is equivalent to $p(A)=100\epsilon$.  We need to
  show that this equation has at least one solution which should be
  bounded from below by a positive constant independent of
  $\epsilon$. This follows from the study of the second order equation
  $q(A)=0$, where $q(A): = A(p(A)/\epsilon -100) =
  5A^2+(\mathcal{O}(\epsilon)-100)A+4\pi$, which has two positive
  solutions $A \geq \frac{1}{8}$ provided that we chose $\epsilon$ in
  Step 1 to be small enough.  Hence by choosing
  $A=A(\tilde{\epsilon})$ to be either solution we have a smoothly
  dependent function on $\tilde{\epsilon}$ for which Moser's equation
  (\ref{moserequality}) holds.  Therefore we may apply the
  non\--compact Moser theorem (Lemma~\ref{ncmoser}) to get a
  diffeomorphism $ \varphi_{\epsilon} \colon \Sigma(\tilde{\epsilon})
  \to \Sigma(\tilde{\epsilon}) $ such that $
  \varphi_{\epsilon}^*(i_{\epsilon}^*\omega_0)=\omega_{\Sigma(\tilde\epsilon)}$
  and therefore by composing $i_{\epsilon}$ with $\varphi_{\epsilon}$
  we may assume that (\ref{theone}) is symplectic. This concludes Step
  2.
  \\
  \\
  \emph{Step 3} (\emph{Preparatory cut\--off functions}). Choose a
  smooth cut\--off function $\chi_{\epsilon} \colon \mathbb{R} \to
  [0,1]$ which is non decreasing on $\R^-$, non increasing on $\R^+$,
  taking values as follows~:
  \begin{eqnarray}
    \chi_{\epsilon} \equiv  1  & \text{ on }&\,\,\,\,\,\, [-a,\,a];  \label{cutoff} \\
    \chi_{\epsilon} \equiv  0  & \text{ on }&\,\,\,\,\,\, \mathbb{R} \setminus [-A+\epsilon^2,\,
    A-\epsilon^2], \label{cutoff2}
  \end{eqnarray}
  and such that it satisfies the following bounds~:
  \begin{itemize}
  \item For every $x \in \mathbb{R}$,
    \begin{eqnarray} \label{slopebound} |\chi'_{\epsilon}(x)| \leq
      \frac{1}{A}+ \epsilon.
    \end{eqnarray}
  \item For every $x \in
    [-A+\frac{\epsilon}{2},\,A-\frac{\epsilon}{2}]$,
    \begin{eqnarray}
      \biggl\rvert \chi_{\epsilon}(x) - \Big(1-\frac{|x|}{A}\Big) \biggr\rvert  \leq \epsilon. \label{trickyinequality}
    \end{eqnarray}
  \end{itemize}
  Such a function $\chi_{\epsilon}$ is depicted in
  Figure~\ref{stepfunction}. The ${\rm C}^0$
  estimate~\eqref{trickyinequality} follows from
  $d=\frac{\epsilon}{200A}\leq \epsilon$ (recall that $A\geq
  \frac{1}{8}$). The ${\rm C}^1$ estimate~\eqref{slopebound} follows from
  the fact that the maximum slope of the graph in
  Figure~\ref{stepfunction} is $\dfrac{1}{A-\frac{\epsilon}{100}}$ and
  that, when $\epsilon<A$,
\[
\frac{1}{A-\tfrac{\epsilon}{100}} \leq \frac{1}{A}+\epsilon.
\]
    \begin{figure}[h]
    \centering
    \includegraphics[width=.6\textwidth]{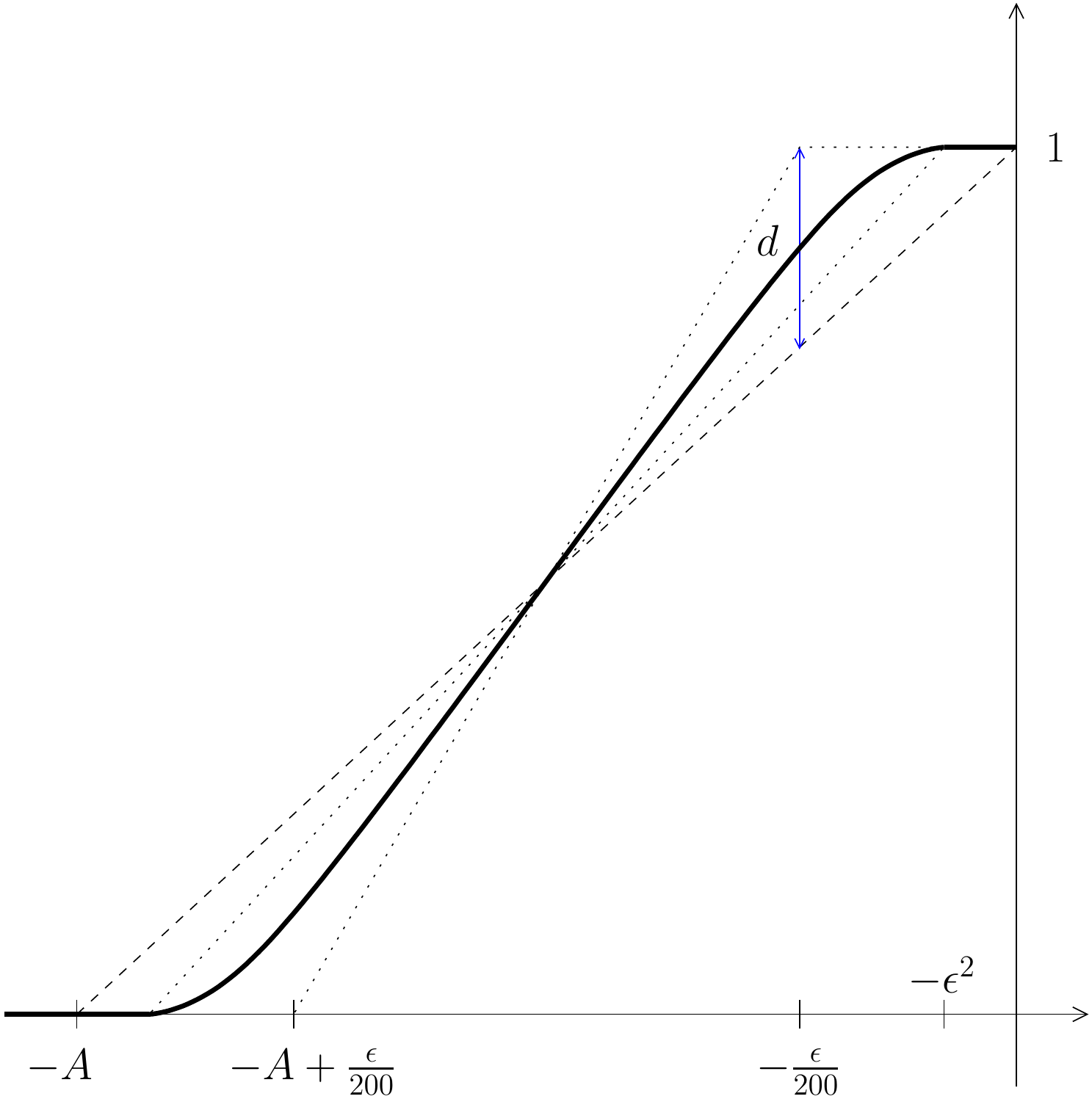}
    \caption{The cut\--off function $\chi_{\epsilon}$. We have
      represented the $x<0$ part; the function is symmetric with
      respect to $x=0$.}
    \label{stepfunction}
  \end{figure}
  This concludes Step 3.
  \\
  \\
  \emph{Step 4} (\emph{The smooth map $\mathcal{I}_{\epsilon}$}).  On
  $\mathbb{R}^2 \times \mathbb{R}^2$ we define the smooth family of
  Hamiltonian functions $(\mathcal{H}_{\epsilon}(x_1,y_1,x_2,y_2):=-
  \chi_{\epsilon}(x_1) x_2 \, \sqrt{\pi})_{\epsilon}$ whose time\--$1$ flows are given by
  the smooth family $(\Phi_{\epsilon})_{\epsilon}$~:
  \begin{eqnarray} \label{equ:flow} \,\,\,\,\,\,\,\,\,\,
    \Phi_{\epsilon}(x_1,y_1,x_2,y_2)=\Big(x_1,\,\,y_1+\chi'_{\epsilon}(x_1)x_2\sqrt{\pi},\,\,x_2,\,\,y_2+
    \chi_{\epsilon}(x_1)\sqrt{\pi} \Big).
  \end{eqnarray}
  Let ${\rm Q}(\sqrt{\pi})$ denotes the open square $ (0,\sqrt{\pi})
  \times (0,\sqrt{\pi}) $ and ${\rm R}(\sqrt{\pi},2\sqrt{\pi})$ be the
  open rectangle $(0,\sqrt{\pi}) \times (0,2\sqrt{\pi})$.  Let
  $\mathcal{S}_{\epsilon}$ be the connected subset of
  $\Sigma(\tilde{\epsilon})$ that is mapped to the horizontal strip $
  \widetilde{\mathcal{S}_{\epsilon}}=(-A,A) \times
  (-\frac{\epsilon}{2},\frac{\epsilon}{2}) $ by the immersion
  $i_{\epsilon}$ (See Figure \ref{fig:fancyimmersion}).  We define
  $\mathcal{I}_{\epsilon} \colon \Sigma(\tilde{\epsilon}) \times {\rm
    Q}(\sqrt{\pi}) \to \mathbb{R}^4 $ by
  \begin{equation}
  \mathcal{I}_{\epsilon}(\sigma,\,b):=
  \begin{cases}
    \Phi_{\epsilon}(i_\epsilon(\sigma),b)  & \text{ if } \sigma \in \mathcal{S}_{\epsilon};\\
    (i_{\epsilon}(\sigma),b)   & \text{ if } \sigma \notin \mathcal{S}_{\epsilon}.\\
  \end{cases}
\label{equ:map-I}
\end{equation}
  Since $0\leq \chi_\epsilon\leq 1$, the image of
  $\mathcal{I}_{\epsilon}$ lies in the set $\mathbb{R}^2 \times {\rm
    R}(\sqrt{\pi},2\sqrt{\pi})$. Moreover, $\mathcal{I}_{\epsilon}$ is
  smooth because the Hamiltonian flow $\Phi_{\epsilon}$ in
  (\ref{equ:flow}) is the identity near $x=\pm A$, $y \in
  (-\frac{\epsilon}{2},\frac{\epsilon}{2})$ (since $\chi_{\epsilon}=0$
  there by (\ref{cutoff2})). For the same reason,
  $\mathcal{I}_\epsilon$ is a local diffeomorphism, since $i_\epsilon$
  is a local diffeomorphism and $\Phi_\epsilon$ is a diffeomorphism.
  \\
  \\
  \emph{Step 5} (\emph{$\mathcal{I}_{\epsilon}$ is injective}). Assume
  $\mathcal{I}_{\epsilon}(\sigma,b)=\mathcal{I}_{\epsilon}(\sigma',b')$.
  There are three cases.
  \begin{itemize}
  \item[(a)] \emph{Suppose that $\sigma \in \mathcal{S}_{\epsilon}$
      and $\sigma' \in \mathcal{S}_{\epsilon}$}.  Then $
    \Phi_{\epsilon}(i_{\epsilon}(\sigma),b)=\Phi_{\epsilon}(i_{\epsilon}(\sigma'),b'),
    $ so since $\Phi_{\epsilon}$ is a diffeomorphism,
    $(i_{\epsilon}(\sigma),b)= (i_{\epsilon}(\sigma'),b')$.  Since
    $i_{\epsilon}|_{\mathcal{S}_{\epsilon}}$ is injective, we have
    that $\sigma=\sigma'$ and $b=b'$ as we wanted.

  \item[(b)] \emph{Suppose that $\sigma \notin \mathcal{S}_{\epsilon}$
      and $\sigma' \notin \mathcal{S}_{\epsilon}$}. Then
    $(i_{\epsilon}(\sigma),b)=(i_{\epsilon}(\sigma'),b')$ and since
    $i_{\epsilon}$ is injective outside of $\mathcal{S}_{\epsilon}$,
    we have $\sigma=\sigma'$ and $b=b'$.

  \item[(c)] \emph{Suppose that $\sigma \in \mathcal{S}_{\epsilon}$
      and $\sigma' \notin \mathcal{S}_{\epsilon}$}.  We have that $
    \Phi_{\epsilon}(i_{\epsilon}(\sigma),b)=(i_{\epsilon}(\sigma'),b').
    $ Let us write in coordinates $
    (x_1,y_1,x_2,y_2)=(i_{\epsilon}(\sigma),b)$ and
    $(x'_1,y'_1,x'_2,y'_2)=(i_{\epsilon}(\sigma'),b')$. {From}~\eqref{equ:flow}
    we have
    \begin{eqnarray}
      \begin{cases}
        x'_1=x_1;\\
        x'_2=x_2; \\
        y'_1=y_1+\chi_{\epsilon}'(x_1)x_2 \sqrt{\pi};\\
        y'_2=y_2+\chi_{\epsilon}(x_1)\sqrt{\pi}. \label{last} \\
      \end{cases}
    \end{eqnarray}
    In particular, $|y'_2-y_2|=\chi_{\epsilon}(x_1)\sqrt{\pi}.$ Since
    $y_2 \in (0,\sqrt{\pi})$ and $y'_2 \in (0,\sqrt{\pi})$ we must
    have that $\chi_{\epsilon}(1)<1$. Hence $|x_1|>a$, and we are
    outside of the vertical strip $|x_1|\leq a$. If  $x_1<-a$, the second to
    last equation in (\ref{last}), and the slope bound
    (\ref{slopebound}), imply
    \begin{eqnarray}
      \begin{cases}
        y_1 \leq y'_1 < y_1 + \bigl( \frac{1}{A} +\epsilon  \bigr) \pi;\\
        x_1 \geq -A \,\,\,\,\,\, (\textup{because}\,\,\, \sigma \in
        \mathcal{S}_{\epsilon}).  \label{impossible}
      \end{cases}
    \end{eqnarray}
    It follows from Figure \ref{fig:fancyimmersion} that
    (\ref{impossible}) is not possible.

    Similarly, if $x_1>a$ then $\chi'\leq 0$, and we have that $ y_1
    \geq y'_1 > y_1 - \bigl( \frac{1}{A}+ \epsilon \bigr)
    \pi, $ which is, again by Figure \ref{fig:fancyimmersion},
    impossible.
  \end{itemize}
  This concludes Step 5.
  \\
\\
 \emph{Step 6} (\emph{Conclusion}).  We have so far shown
  that we have a smooth embedding
  \begin{eqnarray} \label{inj:map} \mathcal{I}_{\epsilon} \colon
    \Sigma(\tilde{\epsilon}) \times {\rm Q}(\sqrt{\pi}) \to
    \mathbb{R}^2 \times {\rm R}(\sqrt{\pi},\, 2 \sqrt{\pi})
  \end{eqnarray}
  for sufficiently small values of $\epsilon>0$.  {From} the formula
  (\ref{equ:flow}) for the flow $\Phi_{\epsilon}$ we have that $
  \pi_1( \mathcal{I}_\epsilon(\Sigma(\tilde{\epsilon}) \times {\rm
    Q}(\sqrt{\pi}))) \subset D_{\epsilon}, $ where $D_{\epsilon}$ is
  depicted in Figure \ref{fig:sr}, and $\pi_1 \colon \mathbb{R}^2
  \times \mathbb{R}^2 \to \mathbb{R}^2$ is the projection onto the
  first factor. So $\mathcal{I}_{\epsilon}$ gives an embedding
  \begin{eqnarray} \label{zeromap} \mathcal{I}_{\epsilon} \colon
    \Sigma(\tilde{\epsilon}) \times {\rm Q}(\sqrt{\pi})
    \hookrightarrow D_{\epsilon} \times {\rm R}(\sqrt{\pi},\, 2
    \sqrt{\pi}).
  \end{eqnarray}
  Let $(\varphi_{\epsilon} \colon D_{\epsilon} \hookrightarrow {\rm
    B}^2(r(\epsilon)))_{\epsilon>0} $ be a smooth family of symplectic
  embeddings, where
  \begin{eqnarray} \label{radius0} r(\epsilon)=\sqrt{\frac{{\rm
          Area}(D_{\epsilon})}{\pi}}.
  \end{eqnarray}     
  Such family exists by Moser's argument (Lemma~\ref{ncmoser}). By
  construction of $D_{\epsilon}$ (Figure \ref{fig:sr})
  there exists a constant $\tilde{c}>0$ such that $r(\epsilon)\leq \sqrt{2} +
  \tilde{c}\epsilon$, for sufficiently small $\epsilon>0$.
  Again by Moser's argument there are symplectomorphisms
  \begin{eqnarray} \label{firstmap} 
  f \colon {\rm R}(\sqrt{\pi},\, 2
    \sqrt{\pi}) \to {\rm B}^2(\sqrt{2})
  \end{eqnarray} 
  and
  \begin{eqnarray} \label{secondmap} g \colon {\rm B}^2(1) \to {\rm
      Q}(\sqrt{\pi}).
  \end{eqnarray}
  We may combine the maps (\ref{zeromap}), (\ref{firstmap}), and
  (\ref{secondmap}) to get a smooth family of symplectic embeddings
  $(I_{\tilde\epsilon})_{\tilde\epsilon>0}$, for $\tilde\epsilon$
  sufficiently small, defined as follows
  \begin{eqnarray}
    \Sigma (\tilde{\epsilon}) \times {\rm B}^2(1) 
    \stackrel{{\rm Id} \otimes g}{\hookrightarrow} 
     \Sigma(\tilde{\epsilon}) \times {\rm Q}(\sqrt{\pi})
   \stackrel{\mathcal{I}_{\epsilon}}{\hookrightarrow} D_{\epsilon} \times {\rm R}(\sqrt{\pi},2\sqrt{\pi})  \nonumber \\
    \stackrel{{\rm Id} \otimes f}{\hookrightarrow} 
    D_{\epsilon} \times {\rm B}^2(\sqrt{2}) \stackrel{\varphi_{\epsilon} \otimes {\rm Id}}{\hookrightarrow} {\rm B}^2(\sqrt{2}+
    c \tilde\epsilon) \times {\rm B}^2(\sqrt{2}), \nonumber
  \end{eqnarray}
  where $c=\tilde c/100$.  This concludes the proof.\end{proof}
\begin{figure}[h]
  \centering
  \includegraphics[width=0.3\textwidth]{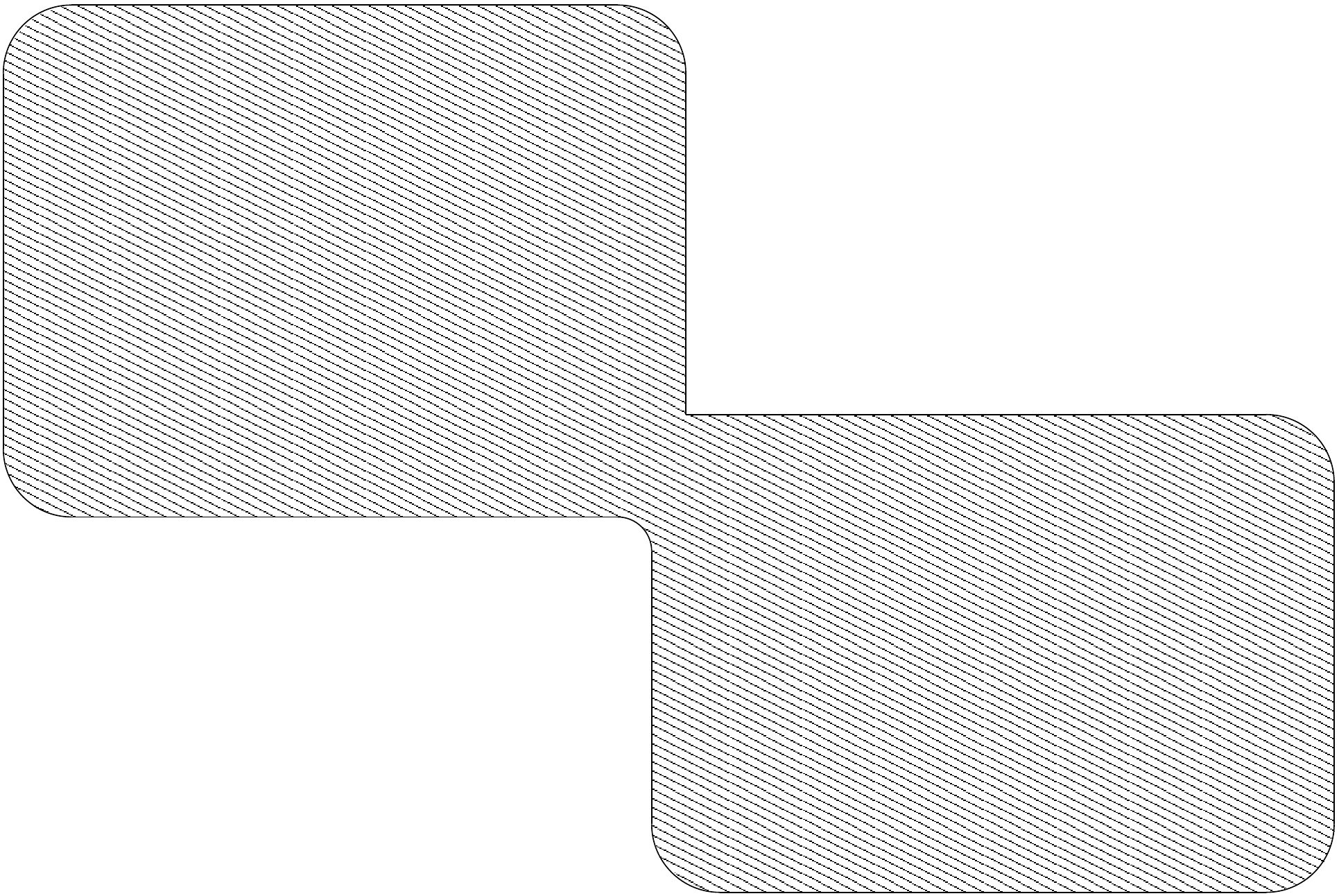}
  \caption{The open set $D_{\epsilon}$ is the envelope of the image of
    the immersion $i_\epsilon$ in
    Figure~\ref{fig:fancyimmersion}. Hence the total area of
    $D_{\epsilon}$ is of order $2(\tfrac{\pi}{A} \times A) +
    \mathcal{O}(\epsilon)= 2\pi +  \mathcal{O}(\epsilon)$.}
  \label{fig:sr}
\end{figure}

Next we prove that \cite[Theorem~1.6]{HiKe2009} holds for smooth families.

\begin{theorem}\label{kh2} 
  Let $n\geq 3$.  There exist constant $C,C'>0$ and a smooth family of
  symplectic embeddings
  \[
  i_{S,\,R} \colon {\rm B}^2(1) \times {\rm B}^{2(n-1)}(S)
  \hookrightarrow {\rm B}^2(R) \times {\rm B}^2(\sqrt 2) \times
  \textup{B}^{2(n-2)}({\textstyle\frac{C S^2}{\sqrt{R-\sqrt 2}}}),
  \] where $(S,R)$ vary in the open set
  \begin{equation}
    \{ (S,R)\in \R^2\, | \qquad S>0, \quad \sqrt 2 < R < \sqrt 2 + C'S^2\}.
    \label{equ:domain}
  \end{equation}
\end{theorem}

\begin{proof}
  Consider the symplectic embedding
  \[i_{T} \colon {\rm B}^{2(n-1)}(T) \hookrightarrow \Sigma \times
  {\rm B}^{2(n-2)}(10T^2), \qquad T>{\textstyle\frac{1}{3}}
  \]
  given by Lemma~\ref{pp:10}, where $\Sigma=(\R^2\setminus \Z^2)/ \Z^2$
  is equipped with the standard quotient symplectic form. For
  $\epsilon>0$, let $\tau_{\sqrt\epsilon}:\R^{2(n-1)}\to\R^{2(n-1)}$
  be the dilation $\tau_{\sqrt{\epsilon}}(x)=\sqrt\epsilon x$. The
  corresponding quotient map $\bar\tau_{\sqrt\epsilon}$ maps
  $\Sigma\times \R^{2(n-2)}$ to $\Sigma(\epsilon) \times \R^{2(n-2)}$.
  
  The map $\bar\tau_{\sqrt\epsilon}\circ i_T \circ
  (\tau_{\sqrt\epsilon})^{-1}$ is a symplectic embedding of
  $\tau_{\sqrt\epsilon}({\rm B}^{2(n-1)}(T))={\rm
    B}^{2(n-1)}(\sqrt\epsilon T)$ into
 $$
 \bar\tau_{\sqrt\epsilon}(\Sigma \times {\rm B}^{2(n-2)}(10T^2)) =
  \Sigma(\epsilon) \times {\rm B}^{2(n-2)}(10\sqrt\epsilon T^2)). 
  $$
  Of
  course, as $T>\frac{1}{3}$ and $\epsilon>0$ vary, the corresponding
  family of embeddings is smooth. By composing with the embeddings
  given by Theorem \ref{pp:11}, we obtain a smooth family of
  symplectic embeddings~:
  \begin{gather*}
    {\rm B}^2(1) \times {\rm B}^{2(n-1)}(\sqrt\epsilon T)
    \hookrightarrow {\rm B}^2(\sqrt 2 + c\epsilon) \times {\rm
      B}^2(\sqrt 2) \times
    \textup{B}^{2(n-2)}(10\sqrt\epsilon T^2),\\
    T>1/3, \quad \epsilon>0.
  \end{gather*}
  The conclusion of the theorem is obtained by the smooth change of
  parameters $(S,R):=(\sqrt\epsilon T,\sqrt 2 + c\epsilon)$, whose
  image is the domain given by~\eqref{equ:domain}, with $C'=9c$. This
  change gives the constant $C=10\sqrt{c}$.
\end{proof}

\section{\textcolor{black}{Proof of Theorem~\ref{answers}}} \label{sec:proof}

Hind and Kerman proved \cite[Theorem 1.5]{HiKe2009} that for any $0 <
R_1<\sqrt{2}$ and any $R_2\geq R_1$ there are no symplectic embeddings
of ${\rm B}^2(1) \times {\rm B}^{2(n-1)}(S) $ into ${\rm B}^2(R_1)
\times {\rm B}^2(R_2) \times \mathbb{R}^{2(n-2)}$ when $S$ is
sufficiently large. Therefore, in order to prove
Theorem~\ref{answers}, it is sufficient to show that ${\rm B}^2(1)
\times \R^{2(n-1)} $ symplectically embeds into ${\rm
  B}^2(\sqrt2) \times {\rm B}^2(\sqrt2) \times \mathbb{R}^{2(n-2)}$.

By Theorem \ref{kh2} there exist constants $C,C'>0$ and a smooth family
of symplectic embeddings $ i_{S,\,R} \colon {\rm B}^2(1) \times {\rm
  B}^{2(n-1)}(S) \hookrightarrow {\rm B}^2(R) \times {\rm B}^2(\sqrt
2) \times \textup{B}^{2(n-2)}({\textstyle\frac{C S^2}{\sqrt{R-\sqrt
      2}}}), $ where $(S,R)$ vary in the set $ A $ of points $(S,R)\in
\R^2$ such that $S>0$ and $\sqrt 2 < R < \sqrt 2 + C'S^2$.  Let
$j_{R,S}$ be the symplectic rescaling~:
$$
{\rm B}^2(\sqrt{2}/R) \times {\rm B}^{2(n-1)}(\sqrt{2}S/R)
\hookrightarrow {\rm B}^2(\sqrt{2}) \times {\rm B}^2(2/R) \times
\textup{B}^{2(n-2)}({\textstyle\frac{C \sqrt{2}S^2}{R\sqrt{R-\sqrt 2}}}),
  $$
  given by $x \mapsto \sqrt{2}/R\,\, i_{S,R}(Rx/\sqrt{2})$.  The
  family $(j_{R,S})_{(R,S) \in A}$ is again a smooth family of
  symplectic embeddings.
  
Consider the smooth subfamily 
 \begin{eqnarray} \label{phifamily} \phi_{\epsilon}:=j_{S,R},
    \,\,\,\, \textup{with}\,\,\,\,\,
    S:=\frac{1}{\epsilon(1-\epsilon)},\,\,\,\,\,\,\,\,\,\,\,
    R:=\frac{\sqrt{2}}{1-\epsilon}.
  \end{eqnarray}
A computation shows that $(R,S)\in A$ as long as
\[
\epsilon^3(1-\epsilon)<\frac{C'}{\sqrt 2},
\]
which holds if $\epsilon<\epsilon_0$ and $\epsilon_0<1$ is small
enough; hence the family (\ref{phifamily}) is well defined, gives symplectic
embeddings from $ {\rm B}^2(1-\epsilon) \times {\rm
  B}^{2(n-1)}(1/\epsilon)$ to ${\rm B}^2(\sqrt{2}) \times {\rm
  B}^2(\sqrt{2}(1-\epsilon)) \times
\textup{B}^{2(n-2)}(\rho(\epsilon))$ with
\[
\rho(\epsilon) := \frac{2^{-1/4}C}{\sqrt{\epsilon^5(1-\epsilon)}}.
\]
Of course, such a function $\rho \colon (0,\epsilon_0) \to (0,\infty)$ is
continuous and
\begin{align*}
  \phi_{\epsilon}\left({\rm B}^2(1-\epsilon) \times {\rm
    B}^{2(n-1)}(1/\epsilon)\right) & \subset  {\rm B}^2(\sqrt{2}) \times
  {\rm B}^2(\sqrt{2}) \times \textup{B}^{2(n-2)}(\rho(\epsilon)) \\
  & \subset {\rm B}^{2n}(2\sqrt{2}+\rho(\epsilon)).
\end{align*}
Thus, in view of Remark \ref{newest} we may apply Theorem \ref{cor} to
the family of symplectic embeddings (\ref{phifamily}) with target
manifold $M={\rm B}^2(\sqrt{2}) \times {\rm B}^2(\sqrt{2}) \times
\R^{2(n-2)}$ as in Definition \ref{defi:smooth}.  In this way we get a
symplectic embedding $ j \colon {\rm B}^2(1) \times \R^{2(n-1)}
\hookrightarrow {\rm B}^2(\sqrt{2}) \times {\rm B}^2(\sqrt{2}) \times
\mathbb{R}^{2(n-2)}, $ as desired, thus proving Theorem~\ref{answers}.

\vspace{1mm}

{\small \emph{Acknowledgements}. We are thankful to Lev Buhovski for
  many fruitful discussions which have been important for the paper.
  We thank also Leonid Polterovich for helpful discussions concerning
  Moser's theorem.  We thank Helmut Hofer for helpful comments on the
  introduction to the paper.  AP learned of the problem treated in this paper in
  a lecture by Helmut Hofer in the Winter of 2011 at Princeton
  University, and he is grateful to him for fruitful discussions and
  encouragement.  AP was partly supported by NSF Grant DMS-0635607 and
  an NSF CAREER Award.  VNS is partially supported by the Institut
  Universitaire de France, the Lebesgue Center (ANR Labex LEBESGUE),
  and the ANR NOSEVOL grant.  He gratefully acknowledges the
  hospitality of the IAS. }

{\small
  \noindent
  \\
  {\bf {\'A}lvaro Pelayo} \\
  School of Mathematics \\
  Institute for Advanced Study\\
  Einstein Drive,   Princeton, NJ 08540 USA \\
  \\
  \noindent
  Washington University,  Mathematics Department \\
  One Brookings Drive, Campus Box 1146 \\
  St Louis, MO 63130-4899, USA.\\
  {\em E\--mail}: \texttt{apelayo@math.wustl.edu} \\
   {\em E\--mail}: \texttt{apelayo@math.ias.edu} 
  \noindent
  \\
  \\
  \noindent
  {\bf San V\~u Ng\d oc} \\
  Institut Universitaire de France
  \\
  \\
  Institut de Recherches Math\'ematiques de Rennes\\
  Universit\'e de Rennes 1, Campus de Beaulieu\\ F-35042 Rennes cedex, France\\
  {\em E-mail:} \texttt{san.vu-ngoc@univ-rennes1.fr}\\

\end{document}